\documentclass[11pt]{amsart}
\usepackage{amssymb}
\usepackage{amsmath}
\usepackage{mathrsfs}
\usepackage{bbm}
\usepackage[dvips]{graphicx}
\usepackage{hyperref}
\usepackage{color}
\usepackage{multirow}
\usepackage{tikz}
\usetikzlibrary{arrows}
\usepackage{soul}
\usepackage[normalem]{ulem}

\numberwithin{equation}{section}

\newcommand{\be}{\begin{equation}}
\newcommand{\ee}{\end{equation}}
\newcommand{\ds}{\displaystyle}

\newcommand{\iden}{\mathbbm{1}}

\newcommand{\Z}[1]{\mathbb{Z}_{#1}}
\renewcommand{\r}{|R|}
\newcommand{\Xu}[1]{X^{\text{(u)}}_{#1}}
\newcommand{\Prob}{\mathbb{P}}
\newcommand{\ann}{\mathrm{ann}}

\newcommand{\tauc}{\tau_{\text{couple}}}

\newtheorem{thm}{Theorem}[section]
\newtheorem{prop}[thm]{Proposition}
\newtheorem{cor}[thm]{Corollary}
\newtheorem{lem}[thm]{Lemma}
\newtheorem{defn}[thm]{Definition}
\newtheorem{rem}[thm]{Remark}

\theoremstyle{remark}
\newtheorem{example}[thm]{Example}

\title
{Random motion on finite rings, I: commutative rings}
\author{Arvind Ayyer}
\address{Arvind Ayyer, Department of Mathematics, 
Indian Institute of Science, Bangalore  560012, India.}
\email{arvind@iisc.ac.in}

\author{Pooja Singla}
\address{Pooja Singla, Department of Mathematics, 
Indian Institute of Science, Bangalore  560012, India.}
\email{pooja@iisc.ac.in}

\date{\today}

\begin{document}

\begin{abstract}
We consider  irreversible Markov chains on finite commutative rings randomly generated using both addition and multiplication. We restrict ourselves to the case where the addition is uniformly random and multiplication is arbitrary.
 We first prove formulas for eigenvalues and multiplicities of the transition matrices of these chains using the character theory of finite abelian groups. 
The examples of principal ideal rings (such as $\mathbb{Z}_{n})$ and finite chain rings (such as $\mathbb{Z}_{p^k})$ are particularly illuminating and are treated separately. 
We then prove a recursive formula for the stationary probabilities for any ring, and use it to prove explicit formulas for the probabilities for finite chain rings when multiplication is also uniformly random.
Finally, we prove constant mixing time for our chains using coupling. 

\end{abstract}

\maketitle

\section{Introduction}
\label{sec:intro}
Random walks on general groups are an extremely well-studied subject, and even those on finite groups have been explored in great detail, with some results appearing as early as the 1950s \cite{good-1951}.
The subject acquired a life of its own starting with the work of Diaconis and Shashahani \cite{diaconis-shahshahani-1981}, where probabilistic questions were answered by appealing to the representation theory of the symmetric group. 
See \cite{diaconis-1998,saloffcoste-2004} for generalizations in this direction.

A concept more general than a random walk is a Markov chain, wherein the probability of being in a future state depends on the past only through the present state. A random walk is then a Markov chain which has the additional property of reversibility (see Definition~\ref{def:revers}).
In parallel with random walks on groups, there has been a growing interest in Markov chains on finite semigroups and monoids, such as the Markov chain on the symmetric group known as the Tsetlin library \cite{tsetlin-1963,hendricks-1972}. A far-reaching generalization of the latter on hyperplane arrangements \cite{BHR-1999} led to a systematic study of Markov chains on monoids known as left-regular bands \cite{brown-2000}. This has since been extended to a more general class known as {$\mathscr{R}$}-trivial monoids \cite{asst-2015,steinberg-2016}.

In a similar vein, Markov chains on $\Z{n}$ 
\cite{chung-diaconis-graham-1987,hildebrand-1993,bate-connor-2018} and on $\Z{p}^k$ \cite{hildebrand-mccollum-2008, asci-2009a,asci-2009b}
 generated by affine random transformations have also been studied. A generalization in this direction is the recent study of very general Markov chains on modules of finite rings~\cite{ayyer-steinberg-2017}.

In this work, we study Markov chains on finite commutative rings generated simultaneously by both addition and multiplication operations as follows. At each step, we choose either to add or multiply the current state with an element of the ring according to a coin toss. The addition is done according to the uniform distribution on the ring, and multiplication according to an arbitrary distribution. Although we will mostly work on rings with identity, results for rings without identity can also be deduced similarly; see Remark~\ref{rem:rings-without-identity}.
We will be interested in the stationary distribution of these chains and their convergence here.
Results about Markov chains on noncommutative finite rings will appear in a subsequent work \cite{AS-2018}.

The plan of the article is as follows. We will give the basic definitions and summarize the main results in Section~\ref{sec:def}. 
We begin with preliminaries in Section~\ref{sec:prelim}. Readers familiar with the basics of finite commutative rings can skip this section.
In Section~\ref{sec:spectrum}, we will prove a general formula for the eigenvalues (and multiplicities) of the transition matrix of the chain.
This is related to the Markov chains on semigroups stated above; see the discussion after Proposition~\ref{prop:Bn}. 
In Section~\ref{sec:stat-dist}, we will prove the formula for the stationary distribution for general rings. 
Lastly, we will show that our chains mix in constant time in Section~\ref{sec:mixing-zpk}. We will end with related open questions in Section~\ref{sec:open}.
 
\section{Definitions and summary of results}
\label{sec:def}

Let $R$ be a finite commutative ring with identity and let $|R|$ denote its cardinality.
We will define a discrete-time Markov chain $(X_n)_{n \in \mathbb{Z}_+}$ with state space $R$ which uses its ring structure. 
The informal description of the chain is as follows. Suppose we are at a certain state $r \in R$ at some time. At the next time step, we toss a biased coin with Heads probability $\alpha$. If the coin lands Heads, we pick a uniformly random element of 
$R$ and add it to $r$. If it lands Tails, we pick an element from $R$ according to an arbitrary distribution and multiply it to $r$. 

To describe the transition probabilities of this chain more formally, we will define a probability distribution on the product space
\[
\mathcal S_R = \{\, (\star,r)\;|\; \star \in \{\times,+\},\; r \in R \}
\]
as follows. The marginal distribution on $\star$ is given by 
\begin{equation}\label{marg-dist}
\Prob (\star) = \begin{cases}
\alpha & \text{if $\star = +$}, \\
1-\alpha & \text{if $\star = \times$},
\end{cases}
\end{equation}
where $\alpha \in (0,1]$ and the conditional distribution on $R$ is 
\begin{equation}\label{cond-dist}
\Prob(X = r \,|\, \star) = \begin{cases}
\ds \frac{1}{|R|} & \text{if $\star = +$}, \\
\ds \beta_r & \text{if $\star = \times$},
\end{cases}
\end{equation}
where $\beta_r \in [0,1]$ for each $r$ and $\sum_r \beta_r = 1$.
Let $(\star,r)$ be sampled from this distribution.
We then have the following Markov chain $(X_n)_{n \in \mathbb{Z}_+}$ on the state space $R$
given by 
\begin{equation}\label{Zn-evolution}
X_{n+1} = X_n \star r.
\end{equation}
We will also consider this Markov chain where multiplication is also performed in a uniformly random manner. To distinguish the two, we will denote the latter by
 $(\Xu{n})_{n \in \mathbb{Z}_+}$. That is to say,
\begin{equation}\label{Zn-evolution-uniform}
\Xu{n+1} = \Xu n \star r,
\end{equation}
where $\star$ is still chosen according to \eqref{marg-dist}, but the conditional distributions are the same, i.e., 
\[
\Prob(X = r \,|\, \star) = \Prob(X = r) =  \frac{1}{r}.
\]
In other words, the distribution here on $\mathcal S_R$ is a product distribution of $\text{Bernoulli}(\alpha)$ on $\{+,\times\}$ and the uniform distribution on $R$.

Unless explicitly specified, we will be talking about the chain $(X_n)_{n \in \mathbb{Z}_+}$.
For $a,b \in R$, we will denote the probability of making a single-step transition from $a$ to $b$ in $R$ by $\mathbb{P}(a \to b)$.
Let $M_R = (\mathbb{P}(a \to b))_{a,b \in R}$ be the transition matrix of $(X_n)_{n \in \mathbb{Z}_+}$ using some ordering of $R$. 
Thus, $M_R$ is a row-stochastic matrix, that is, a matrix of nonnegative entries whose rows sum to 1. More precisely, let $\iden_{m}$ be the column vector of size $m$ consisting of all 1's and consider the matrix $B_R = (\beta_{a,b})_{a,b \in R}$ with 
$\beta_{a,b} = \sum_{a\,x \,=\, b} \beta_x$. 
Then 
\begin{equation}\label{trans-matrix}
M_R = \frac{\alpha}{\r}\, \iden_{|R|} \iden_{|R|}^{\text{tr}}+ (1-\alpha) \, B_R .
\end{equation}
Roughly, a Markov chain is said to be {\em irreducible} if there is a positive probability to get from any state in the chain to any other state in the future. An irreducible Markov chain is said to be {\em aperiodic} if the greatest common divisor  of the set of return times to any state is 1. See \cite{LevinPeresWilmer} for the precise definitions.
Since each entry of $M_R$ is nonzero, we immediately have the following result.

\begin{prop}
\label{prop:irred}
For $\alpha \in (0,1]$, the Markov chain $(X_n)_{n \in \mathbb{Z}_+}$ is irreducible and aperiodic.
\end{prop}

By standard theory (see, for example, \cite[Theorem 4.9]{LevinPeresWilmer}), it follows that $(X_n)_{n \in \mathbb{Z}_+}$ has a unique stationary distribution (see Definition~\ref{def:stat-dist}) denoted by
$\pi$. 
The stationary probability of an element $x \in R$ will be denoted by $\pi(x)$. 
We will consider $\pi$ as a row-vector ordered in the same basis as for $M_R$.

We are going to be interested in the eigenvalues of $M_R$ and the following result tells us that we only need to consider the semigroup action on $R$ by multiplication. Since the $\beta_r$'s are nonnegative and sum to one, $B_R$ has the largest eigenvalue 1 by the Perron-Frobenius theorem.

\begin{prop}[{\cite[Corollary 3.1]{ding-zhou-2007}}] 
\label{prop:Bn}
Let $\lambda_1 = 1, \lambda_2, \dots, \lambda_{\r}$ be the eigenvalues of $B_R$ counted with 
multiplicity. Then the eigenvalues of $M_R$ are $ \lambda_1 = 1, (1-\alpha)\lambda_2, \dots, (1-\alpha)\lambda_{\r}$ counted with multiplicity. 
\end{prop}

In view of the above proposition, to determine eigenvalues and their multiplicities it is sufficient to consider the random walk on the semigroup $R$ under multiplication. It is well known that eigenvalues of $B_R$ are the same as that of the operator of the semigroup algebra $\mathbb C[R]$ obtained by multiplying on the left by $\sum_{x \in R} \beta_x x$ (see \cite[Section 7]{brown-2000}). The commutativity of $R$ implies that $\mathbb C[R]$ is a basic monoid algebra. Basic semigroup algebras have already been studied by Steinberg \cite{steinberg-2008,steinberg-2016}. For example, Steinberg \cite[Proposition~12.10]{steinberg-2016}  proves that eigenvalues can be determined using the fact that $\mathbb C[R]$ projects onto a commutative inverse monoid algebra. However, in this article we approach the problem differently. In particular, we explore the ring structure of $R$ which enables us to give an easy description of the eigenvalues, their multiplicities, the stationary distribution and the mixing time. 

We now write down the main results. Let $R$ be a finite commutative ring with identity. The group of invertible elements of $R$ is denoted by $U_R$.
For $a \in R$, let $I_a$ denote the principal ideal generated by $a$.
Let $\phi$ be a fixed set of generators of distinct principal ideals of $R$. 

Let $\ann(a) = \{ x \in R \mid xa= 0 \}$ be the annihilator of $a$ in $R$. 
For $a \in R$, let $Q_a = R/\ann(a)$ be the quotient ring, $U_a :=U_{Q_a} = U_R/((1 + \ann(a) \cap U_R))$ be its unit group and $f_a: R \rightarrow Q_a$ be the natural projection map. 

The set of characters of $U_R$, that is, the group of homomorphisms from $U_R$ to $\mathbb C^{\times}$, is denoted by $\widehat{U_R}$. 
For $a \in \phi$, let $\Sigma_a$  be the set of all characters of $U_R$ that are obtained by composing a character of $U_a$ with the natural projection from $U_R$ onto $U_a$,
\[
\Sigma_a = \{ \chi \in \widehat{U_R} \mid \chi((1 + \ann(a))\cap U_R) = 1 \}. 
\]
Define $F_a = f_a^{-1}(U_{a})$. The ring $R$ is finite as well as commutative and therefore it is well known that $U_R$ maps onto $U_a$ (see Proposition~\ref{prop:units-onto-units} for a proof). Therefore, for every $x \in F_a$, there exists a unit $u \in U_R$ such that $f_a(x) = f_a(u)$. Further if $u_1$ and $u_2$ are two such units then for $\chi \in \Sigma_a$, we have $\chi(u_1) = \chi(u_2)$ (see Proposition~\ref{prop:Sa-acting-set}). This is the context in which we require units associated with $x \in F_a$. Henceforth for $x \in F_a $ we fix a unit, denoted $u_a(x)$, such that $f_a(u_a(x)) = f_a(x)$. 

We are now in a position to describe the spectrum of the matrix $B_R$. From here, the spectrum of the transition matrix $M_R$ is easily obtained by using Proposition~\ref{prop:Bn}. 

\begin{thm}
\label{thm:spectrum-general} 
For every $\chi \in \Sigma_a$, we obtain an eigenvalue $\lambda_{\chi}$ of $B_R$ given by 
\[
\lambda_{\chi} = \sum_{x \in F_a} \beta_x \chi(u_a(x)). 
\]
Conversely, every eigenvalue of $B_R$ is of the form $\lambda_{\chi}$ for some $\chi \in \Sigma_a$ for some $a \in R$. 
The algebraic multiplicity, $m(\lambda_\chi)$ of $\lambda_{\chi}$ for $\chi \in \Sigma_a$, is given by 
\[
m(\lambda_\chi) = |\{ b \in \phi \mid  F_b = F_a \,\, \mathrm{and} \,\, \chi \in \Sigma_b \}|. 
\]
\end{thm}

After this work appeared, a generalization of this result was proved in~\cite{ayyer-steinberg-2017}.

The background and proof of Theorem~\ref{thm:spectrum-general} will be presented in Section~\ref{sec:spectrum}. From this, the results for principal ideal rings (see Definition~\ref{def:prin ideal ring}) in Corollary~\ref{cor:spectrum-principal} and finite chain rings (see Definition~\ref{def:finite chain ring}) in Corollary~\ref{cor:spectrum-finite-chain-rings} will follow.

We now describe probabilistic aspects of this chain.
The first result is about the stationary distribution. 
For $x,y \in R$ such that $I_x \subseteq I_y$, denote $U_{y,x}$ as the subgroup 
$((1 + \ann(x))\cap U_R)/((1 + \ann(y))\cap U_R)$ of $U_y$. 
Recall that $\beta_{a,b} = \sum_{a\,x \,=\, b} \beta_x$. 

\begin{thm}
\label{thm:stat-dist}
Let $R$ be a finite ring.
The stationary distribution $\pi(x)$ for $x \in R$ of the chain $(X_n)_{n \in \mathbb{Z}_+}$ is given by
\[
\pi(x) = \frac{\ds \frac{\alpha}{\r} + 
 (1-\alpha) \sum_{y \in \phi, I_x \subsetneq I_y} \frac{|U_y|}{|U_x|}
 \left( \sum_{u \in U_y/U_{y,x}} \beta_{f_y^{-1}(u)y,x} \right) \pi(y)}
{\ds 1- (1-\alpha) \left( \sum_{r \in F_{x}} \beta_r \right) }.
\]
\end{thm}

We also obtain the the stationary distribution of $(\Xu n)_{n \in \mathbb{Z}_+}$
as a corollary.

\begin{cor}
\label{cor:stat-dist}
Let $R$ be a finite ring.
The stationary distribution $\pi(x)$ for $x \in R$ of the chain $(\Xu{n})_{n \in \mathbb{Z}_+}$ is given by
\[
\pi(x) = \frac{\ds \alpha + 
 (1-\alpha) \sum_{y \in \phi, I_x \subsetneq I_y} |U_{y}| |\ann(y)| \pi(y)}
{\ds \r - (1-\alpha) |U_{x}| |\ann(x)| }.
\]
\end{cor}

The formula above can be thought of as a special case of a new formula for the stationary distribution of an arbitrary finite-state Markov chain~\cite{rhodes-schilling-2017}. 

Theorem~\ref{thm:stat-dist} and Corollary~\ref{cor:stat-dist} will be proved in Section~\ref{sec:stat-dist}. 
Computationally, Theorem~\ref{thm:stat-dist} can be used recursively by going upwards along the poset of principal ideals (see Section~\ref{sec:spectrum}). The lowest element of this poset is the set of units, and their stationary probability is given by Corollary~\ref{cor:stat-prob-units}. 
Even for $(\Xu{n})_{n \in \mathbb{Z}_+}$, the stationary probabilities seem to be complicated for general rings. However, they become simpler for local rings (see Definition~\ref{def:local ring}) and are given in Corollary~\ref{cor:stat-dist-local-ring}. They are completely described for finite chain rings in Theorem~\ref{thm:stat-dist-finite-chain}.

The mixing time for a Markov chain gives an estimate of the speed of convergence of the chain to its stationary distribution. See Section~\ref{sec:mixing-zpk} for the precise definitions. 
Let $\epsilon < 1/2$ be a fixed positive constant.

\begin{thm}
\label{thm:mixing-time}
The mixing time of the chain $(X_n)_{n \in \mathbb{Z}_+}$ for the ring $R$ is bounded above by 
\[
t_{\text{mix}}(\epsilon) \leq \frac{\log \epsilon}{\log (1-\alpha)} + 1.
\]
\end{thm}

The following example of the ring $\Z{8}$ should serve to illustrate the 
main results described here.

\begin{example}
\label{eg:Zmod8}
Let $R = \Z{8}$. We will denote elements of the ring with bars to avoid confusion and order the elements using the natural increasing order on the integers $\{\bar 0,\dots,\bar 7\}$. One can check that the multiplicative part $B_R$ of the transition matrix is given by
\begin{equation} \label{Zmod8-br}
\left(\begin{array}{cccccccc}
\beta_{0} + \beta_{1} + \beta_{2} + \beta_{3} & \multirow{2}{*}{$0$} & \multirow{2}{*}{$0$} & \multirow{2}{*}{$0$} & \multirow{2}{*}{$0$} & \multirow{2}{*}{$0$} & \multirow{2}{*}{$0$} & \multirow{2}{*}{$0$} \\
 + \beta_{4} + \beta_{5} + \beta_{6} + \beta_{7} &&&&&&& \\
\beta_{0} & \beta_{1} & \beta_{2} & \beta_{3} & \beta_{4} & \beta_{5} & \beta_{6} & \beta_{7} \\
\beta_{0} + \beta_{4} & 0 & \beta_{1} + \beta_{5} & 0 & \beta_{2} + \beta_{6} & 0 & \beta_{3} + \beta_{7} & 0 \\
\beta_{0} & \beta_{3} & \beta_{6} & \beta_{1} & \beta_{4} & \beta_{7} & \beta_{2} & \beta_{5} \\
\beta_{0} + \beta_{2} & \multirow{2}{*}{$0$} & \multirow{2}{*}{$0$} & \multirow{2}{*}{$0$} & \beta_{1} + \beta_{3}  & \multirow{2}{*}{$0$} & \multirow{2}{*}{$0$} & \multirow{2}{*}{$0$} \\
+ \beta_{4} + \beta_{6} & & & & + \beta_{5} + \beta_{7} \\
\beta_{0} & \beta_{5} & \beta_{2} & \beta_{7} & \beta_{4} & \beta_{1} & \beta_{6} & \beta_{3} \\
\beta_{0} + \beta_{4} & 0 & \beta_{3} + \beta_{7} & 0 & \beta_{2} + \beta_{6} & 0 & \beta_{1} + \beta_{5} & 0 \\
\beta_{0} & \beta_{7} & \beta_{6} & \beta_{5} & \beta_{4} & \beta_{3} & \beta_{2} & \beta_{1}
\end{array}\right),
\end{equation}
and $M_R$ by \eqref{trans-matrix}. The graph of multiplicative transitions is drawn in Figure~\ref{fig:eg-Zmod8}, where each transition of a particular value has been drawn in a distinct colour.

The eigenvalues of $B_R$ are given by Theorem~\ref{thm:spectrum-general}. Since $R$ is a finite chain ring, we can appeal directly to Corollary~\ref{cor:spectrum-finite-chain-rings}. Other than the trivial eigenvalue $1$ with multiplicity one, given by the table
\[
\begin{array}{c|c}
\text{Eigenvalue} & \text{Multiplicity} \\
\hline
\beta_1 + \beta_3 - \beta_5 -\beta_7 & 1 \\
\beta_1 - \beta_3 + \beta_5 -\beta_7 & 2 \\
\beta_1 - \beta_3 - \beta_5 +\beta_7 & 1 \\
\beta_1 + \beta_3 + \beta_5 +\beta_7 & 3
\end{array}
\]

In the special case when $\beta_i = 1/8$ for all $i$, we get eigenvalues
$1/2$ with multiplicity three and $0$ with multiplicity four. This can also be seen from Corollary~\ref{cor:spectrum-finite-chain-rings-uniform}.
This shows that the relaxation time of the Markov chain is $2$.
The stationary probabilities are given by
\begin{align*}
\pi(\bar 0) &= \frac{1}{(1+\alpha)^3}, \\
\pi(\bar4) &= \frac{\alpha}{(1+\alpha)^3}, \\
\pi(\bar2) = \pi(\bar6) &= \frac{\alpha}{2(1+\alpha)^2}, \\
\pi(\bar1) = \pi(\bar3) = \pi(\bar5) = \pi(\bar7) &= \frac{\alpha}{4(1+ \alpha)}.
\end{align*}
\end{example}

\begin{center}
\begin{figure}[htbp!]
\includegraphics[scale=0.5]{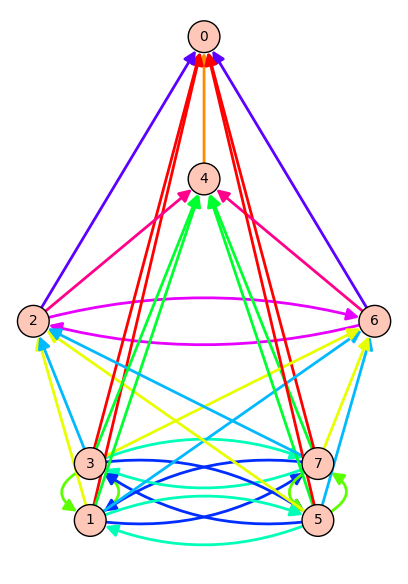}
\caption{The multiplication action of elements in $\Z{8}$. Elements are grouped according to the largest principal ideals they belong to. Transitions with different probabilities are shown in different colours; see \eqref{Zmod8-br} for the values.}
\label{fig:eg-Zmod8}
\end{figure}
\end{center}

\section{Preliminaries}
\label{sec:prelim}
In this section, we collect some basic results on finite commutative rings with identity. In some cases, we will also give short proofs. These results are present in the literature in more general settings (for Artinian rings, for example), and are well-known to specialists. However, they are perhaps easier to state and prove in our setting.
See, for example,~\cite{AtiyahMacdonald1969, bini-flamini-2002}.

\begin{defn}
\label{def:local ring}
A ring $R$ is called {\em local} if it has a unique maximal ideal $\mathrm{m}$. 
\end{defn}

\begin{thm}[{\cite[Theorem~3.1.4, Proposition~3.1.5 and  Lemma~6.4.4]{bini-flamini-2002}}]
\label{thm:structure-of-rings} 
Let $R$ be a finite commutative ring with identity. Then the following are true. 
\begin{enumerate}

\item The ring $R$ satisfies  $R \cong \prod_{i=1}^r R_i $ as rings where each $R_i$ is a finite  local ring with identity. 
\item Given the rings $R$ and $R_i's$ as in (1), the following hold. 
\begin{itemize}
\item[(a)] Every ideal $I$ of $R$ satisfies  $I \cong \prod_{i=1}^r I_i$, where each $I_i$ is an ideal of the ring $R_i.$ 
\item[(b)] The group of units of ring $R$ satisfies, $U_R \cong \prod_{i=1}^r U_{R_i},$ where  $U_{R_i}$ denote the group of units of rings $R_i$ for all $i.$
\end{itemize}
\end{enumerate}
\end{thm}

Throughout the paper, we use the symbol $\setminus$ for set difference.
The following result is standard, but we prove it for completeness.

\begin{lem}
\label{lem:invertible-in-local-ring} 
Let $\mathrm{m}$ be the unique maximal ideal of a finite commutative local ring $R$ with identity $1.$ Then every $r \in R \setminus \mathrm{m}$ is invertible. 
\end{lem}

\begin{proof}
If $r \in R \setminus \mathrm{m}$ is not invertible, then the ideal $(r)$ generated by $r$ is a proper ideal of $R$ and therefore $(r) \subseteq \mathrm{m}$.  This contradicts $r \in R \setminus \mathrm{m}$. Therefore every $r \in R \setminus \mathrm{m}$ must be invertible. 
\end{proof}

\begin{prop}
\label{prop:units-onto-units}
Let $R$ and $S$ be finite commutative rings with identity.
 \begin{enumerate}
 \item Let $f: R \rightarrow S$ be a surjective unital ring homomorphism then $f' = f|_{U_R}$ is a surjective group homomorphism from $U_R$ onto $U_S$.
 \item Let $a, b \in R$ such that $Ra = Rb$. Then there exists $u \in U_R$ such that $ua = b.$
 \end{enumerate}
\end{prop}

\begin{proof} 
Let $R \cong \prod_{i=1}^r R_i $, where each $R_i$ is a finite local ring with identity as given in Theorem~\ref{thm:structure-of-rings}. 
Then $U_R \cong \prod_{i=1}^r U_{R_i}$.
If $I = \mathrm{Ker}(f)$, then by  Theorem~\ref{thm:structure-of-rings}(2), $I \cong \prod_{i=1}^r I_i$. 
Therefore $S \cong \prod_{i=1}^r R_i/(I_i)$  and $U_S \cong \prod_{i=1}^r U_{R_i/(I_i)}$.
To prove (1), it is enough to prove that for each finite local ring $R$ and its ideal $I$, the group $U_R$ maps onto $U_{R/I}$. 
Similarly, (2) follows if the corresponding result for finite local rings is true. Hence, from now onwards, we assume that $R$ is a finite local ring. 

(1): We have that $I = \mathrm{Ker}(f)$ is a proper ideal of the local ring $R$. For any $u_1 \in U_S$, there exists $v_1 \in R$ such that $f(v_1) = u_1$. Let $u_2 \in U_S$ and $v_2 \in R$ such that $f(v_2) = u_2$, $u_2 u_1 = 1$ and  therefore $v_2 v_1 = 1 + x$ for some $x \in I$. Since $I \subseteq \mathrm{m}$, where $\mathrm{m}$ is the unique maximal ideal of $R$, $1 + x \in R \setminus \mathrm{m}$ for every $x \in I$. By Lemma~\ref{lem:invertible-in-local-ring}, $v_2 v_1 \in R \setminus I$ is invertible in $R$ and therefore $v_1$ is also an invertible element of $R$ with $f(v_1 ) = u_1$. Since $u_1 \in U_S$ was chosen arbitrarily  so it follows that $f' = f|_{U_R}: U_R \rightarrow U_S$ is a surjective group homomorphism. 

(2): The hypothesis $Ra = Rb$ implies that there exists $x, y \in R$ such that $a = xb$ and $b = ya$ and therefore $a(1 - xy) = 0$. This implies $xy \in 1 + \ann(a).$ Since $\ann(a)$ is a proper ideal of $R$ and $R$ is local, we have that $xy$ is an invertible element of $R$ by Lemma~\ref{lem:invertible-in-local-ring}. This in particular implies $x$ is an invertible element of $R.$
\end{proof} 

We note that Proposition~\ref{prop:units-onto-units}(2) holds even for finite modules over finite rings. See~\cite[Appendix~A]{ayyer-steinberg-2017} for the proof.

\begin{defn}
\label{def:prin ideal ring}
A commutative ring $R$ with identity is called a {\em principal ideal ring} if every ideal of $R$ is principal. 
We say a ring $R$ is a {\em principal ideal local ring} if it is a principal ideal ring that is also a local ring. 
\end{defn}

\begin{prop}[{\cite{clark-liang-1973,clark-drake-1973,bini-flamini-2002}}]
\label{prop:ideal-structure-PILR} 
Let $R$ be a finite principal ideal local ring with identity and with unique maximal ideal $\mathrm{m}$. Then the following hold. 
\begin{enumerate}
\item Every proper ideal of $R$ is of the form $ \mathrm{m}^k$ (the product of $k$-copies of $\mathrm{m}$) for some $k \in \mathbb N$. 
\item  There exists a smallest $k \in \mathbb N$ such that $\mathrm{m}^t = 0$ if and only if $t \geq k.$
\item There exists $\pi \in R$ such that $\mathrm{m}^k = (\pi^k)$ for every $k \in \mathrm N.$
\item Let $q$ be the cardinality of the residue field $R/\mathrm{m}$ and the $k$ be nilpotency index of $\mathrm{m}$, i.e. $k$ is such that 
$\mathrm{m}^{k-1} \neq 0$ but $\mathrm{m}^k = 0$. Then  $|R| = q^k$ and $|U_R| =(q-1)q^{k-1} .$ 
\end{enumerate}
\end{prop}

Although this result is present in the literature, we include a short proof for the reader's convenience. 

\begin{proof}
Since $R$ is a principal ideal ring, there exists $\pi \in R$ such that $\mathrm{m} = (\pi)$. Therefore it is easy to see that every element of $R$ is of the form $u \pi^t$ for some $u \in U_R$ and $t \in \mathbb N \cup \{ 0\}$. From this (1)-(3) follow easily. For (4), we note that $\mathrm{m}^i /\mathrm{m}^{i+1} \cong R/\mathrm{m}$ for all $1 \leq i \leq k-1$. Therefore the result about $|R|$ and $|U_R|$ follows. 
\end{proof}

\begin{defn}
\label{def:finite chain ring}
A finite commutative ring $R$ with identity is called a {\em finite chain ring} if the set of its ideals form a chain under inclusion.  
\end{defn}
\begin{prop}[{\cite[Theorem~17.5]{Mcdonald74} }]
\label{prop:finite-chain-iff-PILR}
A finite commutative ring is a finite chain ring if and only if it is principal ideal local ring. 
\end{prop}

\section{Eigenvalues and Multiplicities}
\label{sec:spectrum}
The eigenvalues of the transition matrix of a Markov chain give important information about the rate of convergence of the chain to its stationary distribution. 
Suppose $M$ is the transition matrix for a Markov chain $(Y_n)_{n \in \mathbb{Z}_+}$ on the finite state space $\Omega$.
The eigenvalues of $M$ will have their real parts bounded in absolute value by 1. Let us order them in weakly decreasing order of their real parts: $1 = \lambda_1 \geq \Re(\lambda_2) \geq \cdots \geq \Re(\lambda_{|\Omega|}) \geq -1 $. 

\begin{defn}
\label{def:relax-time}
The {\em spectral gap} is given by $\gamma = 1-\Re(\lambda_2)$ and the {\em absolute spectral gap}, by $\gamma_* = 1 - \ds \max(|\Re(\lambda_2)|,|\Re(\lambda_{|\Omega|})|)$. The {\em relaxation time} is given by $t_{\text{rel}} = 1/\gamma_*$.
\end{defn}

The relaxation time is a rough estimate of the time to convergence to the stationary distribution. The mixing time is a more precise estimate, which will be discussed in Section~\ref{sec:mixing-zpk}.
In this section we prove Theorem~\ref{thm:spectrum-general} that describes the eigenvalues of $M_R$ and deduce its corollaries for principal ideal rings and finite chain rings. 

\begin{center}
\begin{figure}[hbp!]
\begin{tikzpicture} [>=triangle 45]
\draw (3.0,0.0) node[circle,inner sep=4pt,draw] (u) {units};
\draw (0.0,3.0) node[circle,inner sep=4pt,draw] (2){$2, 2+t$};
\draw (3.0,3.0) node[circle,inner sep=4pt,draw] (t) {$t, 3t$};
\draw (6.0,3.0) node[circle,inner sep=4pt,draw] (23t) {$2+t, 2+3t$};
\draw (3.0,6.0) node[circle,inner sep=4pt,draw] (2t) {$2t$};
\draw (3.0,9.0) node[circle,inner sep=4pt,draw] (0) {$0$};

\draw [-] (u) -- (2) ;
\draw [-] (u) -- (t) ;
\draw [-] (u) -- (23t) ;

\draw [-] (2) -- (2t) ;
\draw [-] (t) -- (2t) ;
\draw [-] (23t) -- (2t) ;

\draw [-] (0) -- (2t) ;
\end{tikzpicture}

\caption{The Hasse diagram of $\Phi$ of $\mathbb{Z}_4[t]/ \langle t^2 \rangle$.}
\label{fig:eg-poset-prin-ideal}
\end{figure}
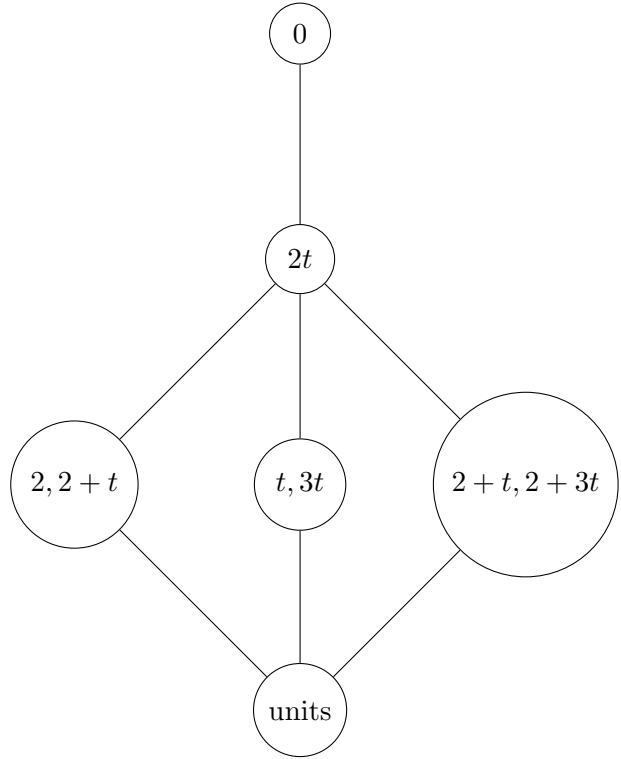
\end{center}

Let $R$ be a finite commutative ring with identity. Recall that for $a \in R$, $I_a$ denotes the principal ideal generated by $a$ and 
$\phi$ is the fixed set of generators of distinct principal ideals of $R$. Moreover, we have an equivalence relation for $a,b \in R$ whenever $I_a = I_b$.
We denote the set of equivalence classes under this relation by $\Phi$. The set $\Phi$ has a natural poset structure, where $a < b$ if $I_b \subsetneq I_a$. 
See Figure~\ref{fig:eg-poset-prin-ideal} for an illustration for the Galois ring $\mathbb{Z}_4[t]/ \langle t^2 \rangle$, which is not a principal ideal ring. In general, this poset is not a lattice, unlike the poset of ideals.
From the definitions of $I_a$ and $\phi$ it is clear that
\[
R = \bigcup_{a \in \phi} I_a.
\]
For a given ideal $I$ of $R$, we consider the set of all ideals $J$ such that $J \subsetneq I$ and the set, 
\[
S_I = I \setminus \mathop{\cup}_{J \subsetneq I } J = \{ x \in I \mid x \notin \mathop{\cup}_{J \subsetneq I } J \} . 
\]
We note that the set $S_I$ is non-empty if and only if $I$ is a principal ideal. For a principal ideal $I$ of a ring $R,$ the set $S_I$ is precisely the set of generators of $I$. 
Whenever $I = I_a$ for some $a \in R$, we write $S_I$ by $S_a$. Therefore, we obtain 
\begin{equation}
\label{R-disjoint union}
R = \mathop{\sqcup}_{a \in \phi} S_a.
\end{equation}
Recall that $U_a$ denotes the group of units of the quotient ring $Q_a = R/\ann(a)$. 
For $a =0$, the ring $Q_0 = U_0$ denotes the zero ring. Further $F_a = f_{a}^{-1}(U_a)$, where $f_a$ is the natural projection map of $R$ onto $Q_a$.  

\begin{lem}
\label{lem:fixing-set}
For any element $x \in R$, the following are equivalent. 
\begin{enumerate}
	\item $xS_a \subseteq S_a$. 
    \item $f_a(x) \in U_a$.
   	\item  $x \in F_a$
\end{enumerate}
\end{lem}

\begin{proof} 
For $a = 0$, the result is true by definition. For nonzero $a$, $xS_a \subseteq S_a$  if and only if there exists $y \in R$ such that $y xa = a$. Now $y xa = a$ if and only if  $yx \in 1 + \ann(a)$. This is equivalent to the fact $f_a(y) f_a(x) = 1$, which in turn is equivalent to $x \in F_a$. Therefore the result follows. 
\end{proof}
	    
\begin{prop}
\label{prop:Sa-acting-set} 
For any $a \in R$, the following are true.
\begin{enumerate}
\item  There exists a 1-1 correspondence between $S_a$ and $U_{a}$ given by $h_a: xa \mapsto f_a(x)$. 
\item For every $x \in F_a$, there exists $u_a(x) \in U_R$ such that 
\begin{equation}
\label{unit-existence}
x z = u_a(x) z \,\, \forall \,\, z \in S_a.
\end{equation}
\end{enumerate}
 Further $u_a(x)$ above is unique in the sense that if $y \in U_R $ satisfies (\ref{unit-existence}) then 
 \[
 \chi(y) = \chi(u_a(x)) \,\, \forall \,\,\chi \in \Sigma_a. 
 \]
\end{prop}

\begin{proof} For $a = 0$, (1) is true by definition and for (2) we take $u_0(x) = 1$ and the rest follows easily. From now on, we assume $a \neq 0$. By Lemma~\ref{lem:fixing-set}, $xa \in S_a$ implies $f_a(x) \in U_a$. Therefore $h_a$ maps $S_a$ to $U_a$ and is injective by the definition of $f_a$.  We have the following short exact sequence of groups:
\[
1 \rightarrow (1 + \ann(a)) \cap U_R \rightarrow U_R \rightarrow U_{a} \rightarrow 1, 
\]
where surjectivity from $U_R$ onto $U_a$ follows by  Proposition~\ref{prop:units-onto-units}. 
By the above short exact sequence and by the definition of $f_a$, for any $z \in U_a$ there exists $u \in F_a \cap U_R$ such that $f_a(u) = z$ and therefore $h_a$ is surjective. For (2), as above there exists, and we fix one, $u_a(x) \in f_a^{-1}(f_a(x)) \cap U_R \subseteq F_a \cap U_R$. It is easy to see that this $u_a(x)$ satisfies (\ref{unit-existence}). For uniqueness, we note that for any $y \in U_R$ such that $ya = u_a(x)a$ implies $y(u_a(x))^{-1} \in 1+ \ann(a)$ and therefore  $\chi(y) = \chi(u_a(x))$ for all $\chi \in \Sigma_a$.
\end{proof}

\begin{rem}
\label{rem:Sa-notation} 
From Proposition~\ref{prop:Sa-acting-set}, the elements of $S_a$ can be written as $ua$ such that $u \in U_R \cap F_a$ with the property that $u a = u'a$ if and only if $f_a(u) = f_a(u') \in U_a$. 
From now on, to simplify notation, whenever there is no ambiguity, we will write elements of $S_a$ by $ua$ for $u \in U_a$.
\end{rem}

Recall for $x,y \in R$ such that $I_x \subseteq I_y$, $U_{y,x}$ is a subgroup of $U_y$ given by $ ((1 + \ann(x))\cap U_R)/((1 + \ann(y))\cap U_R)$. 

\begin{lem} 
\label{lem:Uy-partition}
For $x \in R$, $y \in \phi$, $u_i \in U_y$ and $y_i = u_i y \in S_y$, consider the sets:
\[
P_i = \{r \in R \mid r y_i = x\}.
\]
Then the following are true.
\begin{enumerate}
\item Either $P_i = P_j$ or $P_i \cap P_j = \emptyset$.
\item $P_i = P_j$ if and only if $f_y(u_i u_j^{-1}) \in U_{y,x} \subseteq U_y$. 
\item The relation $y_i \sim y_j$ holds if and only if $P_i = P_j$ partitions $S_y$ into $|U_x|$ classes of size $|U_y|/|U_x|$.
\item $|P_i| = |\ann(y)|$ for all $i$. 
\end{enumerate}
\end{lem}

\begin{proof} 
Let $r \in P_i \cap P_j$, which implies $ru_i y = r u_j y = x$. Then $(1 - u_i u_j^{-1})x = 0$ which is equivalent to saying $u_i u_j^{-1} \in U_{y,x}$. It is also easy to see that if $u_i u_j^{-1} \in U_{y,x}$ and $r \in P_i$ then $r \in P_j$. From this (1) and (2) follow. (3) follows from the fact that $U_{y,x}$ is a subgroup of $U_y$.
Finally, (4) follows from the definitions of $P_i$ and $\ann(y)$. 
\end{proof}

For a given set $T$, we denote $\mathbb C[T]$ as the formal vector space with basis elements parametrized by $T$. In case $T$ is a group (resp. semigroup), we extend the multiplication to $\mathbb C[T]$ and obtain a group algebra (resp. semigroup algebra). We consider $\mathbb C[R]$ as a semigroup algebra with multiplication inherited from that of $R$. As mentioned in the discussion after Proposition~\ref{prop:Bn}, the eigenvalues of $B_R$ are same as that of operator ``left multiplication by $\sum_{x \in R} \beta_x x$'' in the regular representation of the semigroup algebra $\mathbb C[R]$. 
We will use this equivalence to prove Theorem~\ref{thm:spectrum-general}.

\begin{proof}[Proof of Theorem~\ref{thm:spectrum-general}] By \eqref{R-disjoint union}, we have $\mathbb C[R] = \oplus_{a \in \phi} \mathbb C[S_a].$ 
We order the $\mathbb C[S_a]$'s such that 
if $a > b$ in $\phi$, then $\mathbb C[S_a]$ occurs before $\mathbb C[S_b]$.  
Thus, in this ordering, $S_0$ is the first and $S_1$ is the last. We prove that there exists a basis of $\mathbb C[R]$, obtained from those of $\mathbb C[S_a]$, such that $B_R$ is upper triangular in this basis with the required eigenvalues as the diagonal entries.

By Proposition~\ref{prop:Sa-acting-set} for $x \in R$ and $at \in S_a$ we have,
\begin{equation}
\label{ring-action}
x(at) = \begin{cases} u_a(x)at \in S_a & \text{if} \,\, x \in F_a \\  xat \in I_b \subsetneq I_a & \text{if}\,\, x \notin F_a. \end{cases}
\end{equation}
By definition, the set $S_1$ coincides with the group of units of $R$ and therefore there exists a basis say $\mathcal{B}_1 = \{ v_1, v_2, \ldots, v_{|U_R|} \}$ of the group algebra $\mathbb C[S_1]$ such that $v_i$ are eigenvectors under the regular action of $S_1$. This implies that for each $1 \leq i \leq |U_R|$, there exists $\chi_i \in \widehat{S_1}$ such that
\[
u v_i = \chi_i(u) v_i, \,\, \forall \,\,u \in U_R.
\]
We choose a maximal linearly independent subset of $\{ av_1, av_2, \ldots av_{|U_R|} \}$ as a subset of $\mathbb C[R]$. We denote this by $\mathcal{B}_a$. By Proposition~\ref{prop:units-onto-units}, this set is our required basis of the vector space $\mathbb C[S_a]$ for each $a \in \phi$. For any $ u \in S_1$ and $a v_i \in \mathcal B_a$, we have 
\begin{equation}
\label{unit-action}
u a v_i = a u v_i =  \chi_i(u) a v_i.
\end{equation}
Note that for any $(1 + \alpha) \in (1 + \ann(a)) \cap U_R$ we have
\[
av_i = (1 + \alpha) a v_i = \chi_i (1+ \alpha) a v_i
\]
implying that by considering the action of $U_R$ on $\mathbb C[S_a]$ given by (\ref{unit-action}), we obtain only those characters of $U_R$ that belong to $\Sigma_a$. Thus, combining equations (\ref{ring-action}) and (\ref{unit-action}) and the above discussion we obtain that
\[
\sum_{x \in R} \beta_x x(a v_i) = \sum_{x \in F_a} \beta_x \chi_i(u_a(x)) av_i + C
 \] 
where $\chi \in \Sigma_a$, $C \in \mathbb{C}[I_a \setminus S_a]$ and therefore the former belongs to $ \sum_{b > a} \mathbb C[S_b]$. The disjoin union of $\mathcal B_a$ for $a \in \phi$ gives a basis of $\mathbb C[R]$ and from above $B_R$ is upper triangular in this basis with all eigenvalues of the form $\sum_{x \in F_a} \beta_x \chi(u_a(x))$ for some $a \in \phi$ and $\chi \in \Sigma_a$. 

Further we observe that by Proposition~\ref{prop:Sa-acting-set}, the set $S_a = \{xa \mid x \in U_R  \}$ is in bijection with $U_{a}$. Thus the action of $U_R$ on $S_a$ can in fact be viewed as the inflation of the regular action of $U_{a}$ on itself. This implies that every character $\chi \in \Sigma_a$ occurs in the decomposition of $\mathbb C[S_a]$ as a $U_R$-space and that too exactly once. Therefore for $\chi \in \Sigma_a$ and for generic values\footnote{Here generic means that $\{\beta_x\}_{x \in R}$ are chosen off the finite set of hyperplanes where $\lambda_\chi = \lambda_{\chi'}$ for $\chi \neq \chi'$.} 
of $\beta_x$ the algebraic multiplicity of $\lambda_\chi$ is equal to the cardinality of $b \in \phi$ such that $\lambda_\chi$ occurs in the decomposition of $\mathbb C [S_b]$. From the above proof, it follows that $\lambda_\chi$ occurs in the decomposition of $\mathbb C[S_b]$ if and only if $F_a = F_b$ and $\chi \in \Sigma_b$. This justifies the statement about the algebraic multiplicity.  
\end{proof}	

\begin{rem} 
\label{rem:rings-without-identity}
Consider the Markov chain $(X_n)_{n \in \mathbb{Z}_+}$ on a finite commutative ring $R$ without identity. 
Proposition~\ref{prop:Bn} is still valid and so is the fact that eigenvalues of $B_R$ are same as those of operator of $\mathbb C[R]$ described as left multiplication by $\sum_{x \in R} \beta_x$.  Let $m$ be the characteristic of $R$. Consider $\widetilde{R} = R \times \Z m $ as a set with coordinate-wise addition and multiplication given by
\[
(x, a)(y, b) = (bx + ay, ab)
\]
Then $\widetilde{R}$ is a finite commutative ring with identity called the Dorroh extension of $R$ (see \cite{Dorroh-1932}). The ring $R$ embeds into $\widetilde{R}$ as an ideal.  We consider the given probability distribution $\{ \beta_x \}_{x \in R}$ as a probability distribution on $\widetilde{R}$ with its support on $R$.  By restricting this action of $\sum_{x \in R} \beta_x$ on the ideal $\mathbb C[R]$, we can extract the eigenvalues and multiplicities for transition matrix $B_R$ and therefore that of $M_R$. 
\end{rem}

\begin{cor} 
The sum $\sum_{x \in R} \beta_x = 1$ is an eigenvalue of $B_R$ and it occurs with multiplicity one. 
\end{cor}

\begin{proof} The set $\Sigma_0$ consists of only the trivial character of $U_R$ and therefore we obtain that sum $\sum_{x \in R} \beta_x$ is an eigenvalue of $B_R$. Further $F_a = R$ if and only if $a = 0$. This implies our multiplicity result.  
\end{proof}

\begin{cor}
\label{cor:uniform-multiplication-general-spectrum}
If $\beta_x = \frac{1}{|R|}$ for all $x \in R$, 
then the following are true.
\begin{enumerate} 
\item All eigenvalues of $B_R$ are rational.  
\item Any nonzero eigenvalue of $B_{R}$ is equal to $ \frac{|F_a|}{|R|}$ for some $a \in \phi$. 
\item The number of nonzero eigenvalues of $B_{R}$ is equal to the number of distinct principal ideals of $R$. 
\end{enumerate}
\end{cor}

\begin{proof} 
It is clear that (2) implies (1). For (2), let $W \subseteq U_R$ be the set of distinct coset representatives of $(1 + \ann(a)) \cap U_R$ in $U_R$. Then for every $x \in F_a$, there exists a unique $w \in W$ such that $u_a(x)w^{-1} \in 1 + \ann(a)$. Then $\chi \in \Sigma_a$ implies $\chi(u_a(x)) = \chi(w)$. This gives that
\[
 \frac{1}{|R|}\sum_{x \in F_a} \chi(u_a(x)) = \frac{|\ann(a)|}{|R|} \sum_{w \in W} \chi(w)
 \] 
 Further $1 + \ann(a)$ is in the kernel of $\chi$, and therefore $\chi$ can be viewed as character of $U_{a}$ satisfying $\chi(w) = \chi(f_a(w))$. The fact that $W$ consists of coset representatives gives that $f_a(w_1) \neq f_a(w_2)$ for $w_1, w_2 \in W$ whenever $w_1 \neq w_2$. Thus $\sum_{w \in W} \chi(w) = \sum_{y \in U_{a}} \chi(y)$ for a character $\chi$ of $U_{a}$. By Schur's lemma, we have
 \[
  \sum_{y \in U_{a}} \chi(y) = \begin{cases} |U_{a}| & \text{if}\,\, \chi = {\mathbf 1}_{U_{a}} \\
  0 & \text{if} \,\, \chi \neq {\mathbf 1}_{U_{a} }\end{cases}
 \] 
 Now (2) follows by observing that $|F_a| = |U_{a}| |\ann(a)|$. For (3) observe that for each $a \in \phi$, we will have exactly one nonzero eigenvalue given by $|F_a|/|R|$. 
\end{proof}

\subsection{Principal Ideal Rings} 
\label{sec:PIR}
Now we specialize to the case where $R$ is a {\em principal ideal ring} (PIR) defined in Definition~\ref{def:prin ideal ring}. Due to their simpler ideal structure, Theorem~\ref{thm:spectrum-general} specializes considerably.

By Theorem~\ref{thm:structure-of-rings} and Proposition~\ref{prop:ideal-structure-PILR}, for the finite PIR $R$, there exists Principal ideal local rings $R_1, R_2, \cdots, R_r$ such that every $x \in R$ can be written as a tuple $ (x_1, x_2, \ldots, x_r) \in R$ with $x_i \in R_i$ for $1 \leq i \leq r$. In this case, we also denote the element $x$ by $\prod_{i=1}^r x_i$. Let $\mathrm{m}_i$ be the unique maximal ideal of $R_i$ with a fixed generator $\pi_i$. We set $(\mathrm{m}_i)^0 = R_i$. Let $k_i$ be the smallest positive integer such that $\mathrm{m}_i^{k_i-1} \neq 0$ and $\mathrm{m}_i^{k_i} = 0$. In view of Theorem~\ref{thm:structure-of-rings} and Proposition~\ref{prop:ideal-structure-PILR}, every ideal of $R$ is of the form 
\[
\prod_{i = 1}^r (\mathrm{m}_i)^{a_i},\,\, \text{with}\,\, 0 \leq a_i \leq k_i\,\, \text{for all}\,\, 1 \leq i \leq r,
\]
generated by $a = \prod_{i = 1}^r (\pi_i)^{a_i}$.

 Therefore, the set $\phi$ can be identified with the set of elements $\{(\pi_1^{a_1}, \pi_2^{a_2}, \ldots, \pi_r^{a_r}) \mid 0 \leq a_i \leq k_i\}.$ For any $a = \prod_{i=1}^r \pi_i^{a_i} \in \phi$, let 
$s(a) = \{ i \mid a_i \neq k_i \} \subseteq \{1,\dots,r\} $
denote the support of $a$. 
Then we denote $I_a$ by $\prod_{i \in s(a)} \mathrm{m}_i^{a_i}$.
For $T \subseteq \{1,\dots,r\}$, define $R_T$, a subset of $R$, by a set consisting of $x \in R$ such that $x_i \in U_{R_i} $ for $i\in T$. Then $R_{s(a)} = F_a$. Further, for any $x \in F_a$, the associated unit $u_a(x)$ can be easily defined by the following.  
\[
\begin{cases} 
(u_a(x))_i = 1 , & \text{for all} \,\, i \notin s(a), \\
(u_a(x))_i = x_i , & \text{for all} \,\, i \in s(a). 
\end{cases}  
\]
The following definition is important for us. 
\begin{defn}
For a commutative ring $R$ with identity and $\chi \in \widehat{U_R}$, we say that the ideal $I$ is a {\em conductor} of $\chi$, denoted $\mathrm{cond}(\chi)$, if $I$ is the largest ideal of $R$ such that 
\[
\chi((1+ I) \cap U_R) = 1.
\] 

\end{defn}

For principal ideal rings, we obtain the following result. 

\begin{cor} 
\label{cor:spectrum-principal}
Let $R$ be a PIR of the form $R \cong \prod_{i=1}^r R_i $.
For every $\chi \in \Sigma_a$, there exists an eigenvalue $\lambda_{\chi}$ of $B_R$ given by,
	\[
	\lambda_{\chi} = \sum_{x \in R_{s(a)}}{\beta_x}\chi(u_a(x)),
	\]
and	conversely every eigenvalue of $B_R$ is of the form $\lambda_{\chi}$ for some $\chi \in \Sigma_a$ for some $a \in R$. For generic values of $\beta_x$ and for character $\chi \in \Sigma_a$ such that  
\[
 \mathrm{cond}(\chi)=  \prod_{i=1}^r \mathrm{m}_i^{b_i},
 \]
the algebraic multiplicity of $\lambda_\chi$ is $\prod_{i\in s(a)} (k_i - b_i)$. 
\end{cor}

\begin{proof} 
The result about eigenvalues is given by Theorem~\ref{thm:spectrum-general}. For the algebraic multiplicity, we observe that if $\chi \in \Sigma_a$ and $\mathrm{cond}(\chi) = I_b = \prod_{i=1}^r \mathrm{m}_i^{b_i}$ then by definition of conductor $\ann(a) \subseteq I_b$. This in particular implies that $b_i < k_i$ for all $i \in s(a)$ and $b_i = k_i$ for all $i \notin s(a)$. Therefore, by Theorem~\ref{thm:spectrum-general}, the algebraic multiplicity of $\lambda_\chi$ is the same as the cardinality of $c \in \phi$ such that $\ann(c) \subseteq I_b$ and $s(c) = s(a)$. Therefore $I_c = \prod_{i \in s(a)} \mathrm{m}_i^{r_i}$ such that $ b_i \leq k_i - r_i < k_i$ for all $i \in s(a)$. This justifies the result about algebraic multiplicity.       	
\end{proof}
	
The following corollary is a direct consequence of Corollary~\ref{cor:uniform-multiplication-general-spectrum}. 

\begin{cor} 
\label{cor:pir-spec}
Let $\beta_x = 1/|R|$ for all $x \in R$. Then the distinct eigenvalues of $B_R$ are given as follows:
for each $\chi \in \widehat{U_a}$ we have the eigenvalue
\[
\lambda_{\chi} = \begin{cases} \frac{|R_{s(a)}|}{|R|} & \text{ if $\chi = {\bf 1}_{U_a} $}\\
0 & \text{ if $\chi \neq {\bf 1}_{U_a}$} \end{cases}
\]
\end{cor}
Now we specialize to the principal ideal ring $R = \Z n$. 

\begin{cor} 
\label{cor: zn-eigenvalues}
Let $R = \Z{m}$  with $m = p_1^{e_1}p_2^{e_2}\cdots p_r^{e_r}$ where $p_i$'s are distinct primes ($p_1 < \dots < p_r$) and further suppose that $\beta_x = 1/m$ for all $x \in \Z{m}$. Then the following are true. 
\begin{enumerate}
\item The eigenvalues of $B_R$ are given by
\[
\begin{cases} 
\prod_{ i \in T} (1-1/p_i) & \text{ for $\varnothing \neq T \subseteq \{1,\dots,r\}$} \\
1  & \text{ for $T = \varnothing$}.
\end{cases} 
\]
\item The second largest eigenvalue of $B_R$ is $(1-1/p_r)$.
\item The algebraic multiplicity of eigenvalue $\prod_{ i \in T} (1-\frac{1}{p_i})$ is $\prod_{i \in T}{e_i}$. 

\end{enumerate}
\end{cor}

\begin{proof}  From Corollary~\ref{cor:pir-spec} and by the facts that $\Z{m} \cong \prod_{i=1}^k \Z{p_i^{e_i}}$ and $|U_{\Z {p^e}}| = (p-1)p^{e-1}$, we obtain (1) and (2). For (3), we note that if $I_a$ and $I_b$ are ideals of $\Z m$ such that $s(a) \neq s(b)$ then $|R_{s(a)}| \neq  |R_{s(b)}|$. Now result follows by Corollaries~\ref{cor:spectrum-principal} and ~\ref{cor:pir-spec}.  
\end{proof}

\begin{rem}
Consider the chain $(\Xu n)_{n \in \mathbb{Z}_+}$ on $R = \Z{m}$  with $m = p_1^{e_1}p_2^{e_2}\cdots p_r^{e_r}$ where $p_i$'s are distinct primes ($p_1 < \dots < p_r$) and where the multiplication distribution is uniform. By Corollary~\ref{cor: zn-eigenvalues}, the spectral gap of the chain is $1/p_r$ and the relaxation time is $p_r$.
\end{rem}

\subsection{Finite Chain Rings}
\label{sec:finite-chain-ring}
Recall that finite chain rings are given by Definition~\ref{def:finite chain ring} and their important properties are obtained by combining Propositions~\ref{prop:ideal-structure-PILR} and \ref{prop:finite-chain-iff-PILR}.
In this subsection, we give eigenvalues and their algebraic as well as geometric multiplicities of the transition matrix of $B_R$ for finite chain rings $R$.

\begin{cor} 
\label{cor:spectrum-finite-chain-rings}  

Let $R$ be a finite chain ring with length $k$. Let $\mathcal{E}_R$ be the set of eigenvalues of $B_R$. Then $\mathcal{E}_R \setminus \{ 1 \} $ is in one to one correspondence with $\widehat{U_R}$, with bijection from $\widehat{U_R}$ to $\mathcal{E}_R \setminus \{ 1 \} $ given by  
\[
\chi \mapsto \lambda_{\chi} = \sum_{x \in U_R} \chi(x) \beta_x. 
\]
Further, for generic values of $\beta_x$, the geometric multiplicity of $\lambda_\chi$ is one and the algebraic multiplicity of $\lambda_\chi$ is $k-e$ where $e$ is such that $\mathrm{cond}(\chi) = \mathrm{m}^e$ . 
\end{cor}

\begin{proof} In this case, for every $a \in R$ such that $a \neq 0$, we have $\Sigma_a \subseteq \widehat{U_R}$ and $R_{s(a)} = U_R$, where $R_{s(a)}$ is as defined in Section~\ref{sec:PIR}. 
Then the result about the bijective correspondence between $\mathcal{E}_R \setminus {1} $ and $\widehat{U_R}$ and their algebraic multiplicity follows from Corollary~\ref{cor:spectrum-principal}. 
To prove the result about the geometric multiplicity, we follow the notations of the proof of Theorem~\ref{thm:spectrum-general}. Let $\chi$ has conductor $(\pi^e)$. This means that $\mathrm{Ker}(\chi) = 1 + \mathrm{m}^e$. Let $v \in \mathbb{C}[S_1]$ be the unique (up to scalar multiplication) vector such that 
\[
u v = \chi(u) v \,\, \forall \,\, u \in U_R.
\] 
Then by the definition of conductor, we have $\pi^{k-e} v = 0$ and $\pi^{k-e-1} v \neq  0$. Therefore we get that the space generated by $\{\pi^i v \}_{0 \leq i \leq k-e-1}$, say $W$, is the generalized $\lambda_\chi$-eigenspace of dimension $k-e$. We prove that restriction of $(B_R- \lambda_\chi I)|_W$ has index of nilpotency equal to $k-e$. This will prove that geometric multiplicity is equal to one. For this observe that $((B_R - \lambda_\chi I)|_W)^{k-e-1}(v)$ is a scalar multiple of $\pi^{k-e-1} v$ with the scalar being some combination of $\beta_x$. As $\beta_x$'s are generic, this scalar must be nonzero and therefore we have that the index of nilpotency is in fact $k-e$. This proves the result about geometric multiplicity. 
\end{proof}

\begin{cor} 
\label{cor:spectrum-finite-chain-rings-uniform}
Let $R$ be a finite chain ring with length $k$.
When $\beta_x = 1/|R|$ for all $x \in R$, we have exactly three distinct eigenvalues given by $1$, $|U_R|/|R|$, and $0$ with multiplicities one, $k$ and $|R|-(k+1)$ respectively. 
\end{cor}

\begin{proof} 
The result follows from Corollary~\ref{cor:pir-spec} and the observation that in case $R$ is finite chain ring, it has $k$ nonzero ideals and for any nonzero ideal $I_a$ of $R$, we have $R_{s(a)} = U_R$. 
\end{proof}

Now we discuss an example of $\Z{9}$ to make the above ideas clear. 

\begin{example}
\label{eg:Zmod9}
We write the elements of $R = \Z{9}$ by $\{ \bar{0},\dots,\bar{8} \}$, where it is understood that addition and multiplication is modulo $9$. Then $U_R = \{\bar{1}, \bar{2}, \bar{4}, \bar{5}, \bar{7}, \bar{8} \}$. Note that $U_R$ is a cyclic group of order $6$ generated by $\bar{2}$. 
Let $\zeta$ be the sixth primitive root of unity. Define $\chi_i: U_R \rightarrow \mathbb C^{\times}$ by $\chi_i(\bar{2}) = (\zeta^i)$ for $1 \leq i \leq 6$. Then $\chi_i$'s form a complete set of distinct characters of $U_R$. Here $S_0 = U_R$, $S_1 = \{ \bar{3}, \bar{6} \}$, $S_2 = \{\bar{0}\}$.  For $\mathcal{B}_0$, we consider the following vectors in $\mathbb C[S_0]$:
\begin{equation*}
\begin{split}
v_1 &= \bar{2} + \zeta^5 \bar{4}+ \zeta \bar{8}+ \bar{7}+ \zeta^5 \bar{5}+ \zeta \bar{1},  \\
v_2 &= \bar{2} + \zeta^4 \bar{4}+ \zeta^2 \bar{8}+ \bar{7}+ \zeta^4 \bar{5}+ \zeta^2 \bar{1}, \\
v_3 &= \bar{2} + \zeta^3 \bar{4}+ \bar{8}+ \zeta^3 \bar{7}+ \bar{5}+ \zeta^3 \bar{1}, \\
v_4 &= \bar{2}+ \zeta^2 \bar{4}+ \zeta^4 \bar{8}+ \bar{7}+ \zeta^2 \bar{5}+ \zeta^4 \bar{1}, \\
v_5 &= \bar{2}+ \zeta \bar{4}+ \zeta^2 \bar{8}+ \zeta^3 \bar{7}+ \zeta^4 \bar{5}+ \zeta^5 \bar{1}, \\
v_6 &= \bar{2}+ \bar{4}+ \bar{8}+ \bar{7}+ \bar{5}+ \bar{1}. \\
\end{split}
\end{equation*}
 Then it is easy to see that for $u \in U_R$, we have 
 \[
 u v_i = \chi_i(u)v_i \,\, \forall \,\,u \in U_R \,\, \mathrm{and}\,\, 1 \leq i \leq 6.
 \] 
 Since $\chi_i$'s are distinct characters, so the set $\{v_i\}_{1 \leq i \leq 6}$ clearly form an eigenbasis of $C[S_0]$ under the action of $U_R$. For $\mathcal{B}_1$, observe that $\bar{3} v_1$, $\bar{3} v_2$, $\bar{3} v_4$ and $\bar{3} v_6$ are all scalar multiples of each other and $\bar{3} v_3$, $\bar{3} v_5$ are linearly dependent.  
So it is clear that $w_1 = \bar{3} v_3$ and $w_2 = \bar{3} v_6$ form required the basis of $\mathbb C[S_1]$ and we obtain,
\[
\bar 2(w_1) = w_1\,\, ; \,\, \bar 2 (w_2) = - w_2 = \chi^3 w_2.
\] 
Thus the only characters of $U_R$ 
 obtained by its action on $\mathbb C[S_1]$ are $\chi_3$ and $\chi_6$. These are precisely the characters with conductor $(\bar{3})$. Hence the eigenvalues $\sum_{x \in U_R} \beta_x \chi_i(x)$ for $i = 3, 6$ appear  with multiplicity two and the eigenvalues $\sum_{x \in U_R} \beta_x \chi_i(x)$ for $i = 1, 2, 4, 5$ appear with multiplicity one. Finally, by the action of $B_R$ on $\mathbb C[S_2]$ we obtain the eigenvalue $\sum_{x \in R} \beta_x = 1$ and this clearly occurs with multiplicity one. 
\end{example}

\section{The stationary distribution}
\label{sec:stat-dist}
In this section, we will prove the general results for the stationary distributions of $(X_n)_{n \in \mathbb{Z}_+}$ (Theorem~\ref{thm:stat-dist}) and $(\Xu n)_{n \in \mathbb{Z}_+}$ (Corollary~\ref{cor:stat-dist}). 
We will also write down an explicit expression for the stationary probability of units in both chains in Corollary~\ref{cor:stat-prob-units} and Corollary~\ref{cor:stat-prob-units-special} respectively.
We will also deduce the formula for local rings for the chain $(\Xu n)_{n \in \mathbb{Z}_+}$ in Corollary~\ref{cor:stat-dist-local-ring}.
We will give the complete formula for finite chain rings in Section~\ref{sec:stat-dist-finite-chain}. 

We first begin with the relevant definitions. More details can be found, for example, in \cite{LevinPeresWilmer}. 
Let $(Y_n)_{n \in \mathbb{Z}_+}$ be a discrete time Markov chain on the space $\Omega$ with transition matrix $M$.

\begin{defn}
\label{def:stat-dist}
The {\em stationary distribution} of the Markov chain $(Y_n)_{n \in \mathbb{Z}_+}$ is the row-vector $\pi$ satisfying $\pi M = \pi$ whose entries sum to 1.
\end{defn}

\begin{defn}
\label{def:revers}
A Markov chain $(Y_n)_{n \in \mathbb{Z}_+}$ is said to be {\em reversible} if, for any two states $x,y \in \Omega$, its stationary distribution $\pi$ satisfies 
\[
\pi(x) \mathbb{P}(x \to y) = \pi(y) \mathbb{P}(y \to x).
\]
\end{defn}

\begin{prop}
\label{prop:equal-stat-prob}
Let $R$ be a ring and $I$ be a principal ideal in $R$. 
For $a,b \in S_I$, the stationary probabilities of the chain $(X_n)_{n \in \mathbb{Z}_+}$ satisfy $\pi(a) = \pi(b)$.
\end{prop}

\begin{proof}
This follows from the existence of an automorphism $u \in U_R$ from Remark~\ref{rem:Sa-notation} which takes $a \mapsto b = u a$.
Then, for any principal ideal $J$ and any $x \in S_J$, there exists a $y  \in S_J$ (for example, $y = ux$) such that $\beta_{x,a} = \beta_{y,b}$.
\end{proof}

We now prove the formula for the stationary distribution by a recursive argument. A vast generalization of this technique, applicable to any Markov chain, has been recently proposed by Rhodes and Schilling~\cite{rhodes-schilling-2017}.

\begin{proof}[Proof of Theorem~\ref{thm:stat-dist}]
By the uniqueness of the stationary distribution (see Proposition~\ref{prop:irred}), it suffices to solve the so-called master equation,
\begin{equation}\label{master-eq}
\pi(x) = \sum_{y \in R} \mathbb{P}(y \to x) \pi(y).
\end{equation}
Every element $y$ in $R$ can make a transition to $x$ by the addition of $x-y$ with probability $\alpha/|R|$. This is the unique transition by addition. We now split the above sum on the right hand side in two parts according to whether $y$ can make a multiplicative transition to $x$ or not. Let $R_{y,x} = \{r \in R \mid y r = x\}$.
If $I_y \cap I_x \neq I_x$, then there is no such transition and  if $I_x \subseteq I_y$, there is one transition for each element in $R_{y,x}$.
This gives
\[
\pi(x) = 
\sum_{\substack{y \in R \\ I_x \subseteq I_y}} 
\left(\frac{\alpha}{\r} + (1-\alpha) \beta_{y,x} \right) \, \pi(y) 
+ \sum_{\substack{y \in R \\ I_y \cap I_x \neq I_x}} \frac{\alpha}{\r} \,\pi(y).
\]
Combining the first term from the first sum and the second sum gives
\[
\pi(x) = \frac{\alpha}{\r} +
(1-\alpha) \sum_{\substack{y \in R \\ I_x \subseteq I_y}} \beta_{y,x}  \, \pi(y).
\]
We now split the second sum according to whether $I_y$ equals $I_x$ or not. Then, using Proposition~\ref{prop:equal-stat-prob}, we obtain
\begin{equation}
\label{pi-sum}
\pi(x) = \frac{\alpha}{\r} +
(1-\alpha) \left( \pi(x) \sum_{\substack{x' \in R \\ I_x = I_{x'}}}  \beta_{x',x} +  \sum_{\substack{y \in R \\ I_x \subsetneq I_y}} \beta_{y,x}  \, \pi(y), \right).
\end{equation}
For the first sum in \eqref{pi-sum}, when $y \in S_x$, $U_{y,x}$ is trivial. Therefore, the sets $P_i$ in Lemma~\ref{lem:Uy-partition} are disjoint and form a partition of $F_x$, giving
\[
\sum_{\substack{x' \in R \\ I_x = I_{x'}}}  \beta_{x',x} = \sum_{r \in F_x} \beta_r.
\]
Let us now consider the second sum in \eqref{pi-sum}.
By Lemma~\ref{lem:Uy-partition}, parts (1) and (2), we can restrict the $y$-sum to be over $\phi$ and collect coset representatives in $U_y/U_{y,x}$ to account for all the terms. By Lemma~\ref{lem:Uy-partition}(3), the number of times each representative occurs is $|U_y|/|U_x|$. 
We then use Proposition~\ref{prop:equal-stat-prob} to obtain the identity
\[
\sum_{\substack{y \in R \\ I_x \subsetneq I_y}} \beta_{y,x} \, \pi(y)
= \sum_{y \in \phi, I_x \subsetneq I_y} \frac{|U_y|}{|U_x|} \, \pi(y)
 \sum_{u \in U_y/U_{y,x}} \beta_{f_y^{-1}(u)y,x}.
\]
Combining these elements and simplifying leads to the desired result.
\end{proof}

\begin{proof}[Proof of Corollary~\ref{cor:stat-dist}]
From Lemma~\ref{lem:Uy-partition} parts (3) and (4), when $\beta_x = 1/|R|$ for all $x \in R$, we obtain
\[
\sum_{u \in U_y/U_{y,x}} \beta_{f_y^{-1}(u)y,x} = \frac{\ann(y)| |U_x|}{|R|}.
\]
Finally, from the definition of $F_x$, it is clear that $|F_x| = |\ann(x)| |U_x|$, completing the proof.
\end{proof}

Theorem~\ref{thm:stat-dist} and Corollary~\ref{cor:stat-dist} 
can be used to calculate the stationary probability of $x \in R$ using the poset of principal ideals. The difficulty in the calculation depends on the height of $I_x$ in this poset. The easiest stationary probabilities to calculate are those of units, while the hardest is that for the zero element.

\begin{cor}
\label{cor:stat-prob-units}
The stationary probability of $x \in U_R$ in the chain $(X_n)_{n \in \mathbb{Z}_+}$ is given by
\[
\pi(x) = \frac{\alpha}{\r \left( \ds \sum_{y \notin U_R} \beta_y 
+ \alpha \sum_{y \in U_R} \beta_y \right)}.
\]
\end{cor}

\begin{proof}
Since $I_x = R$, the sum in the numerator of Theorem~\ref{thm:stat-dist} is empty and  $F_x = U_R$.
\end{proof}

The following corollary is then immediate.

\begin{cor}
\label{cor:stat-prob-units-special}
The stationary probability of $x \in U_R$ in the chain $(\Xu n)_{n \in \mathbb{Z}_+}$ is given by
\[
\pi(x) = \frac{\alpha}{\r- (1-\alpha)|U_R|}.
\]
\end{cor}

For local rings, Corollary~\ref{cor:stat-dist} simplifies to the following.
\begin{cor}
\label{cor:stat-dist-local-ring}
Let $R$ be a finite local ring.
Then the stationary probability $\pi(x)$ for $x \in R$ in the chain $(\Xu n)_{n \in \mathbb{Z}_+}$is given by
\[
\pi(x) = \frac{\alpha + (1-\alpha)|U_R| \ds \sum_{y \in \phi, I_x \subsetneq I_y} \pi(y)} {\ds \r - (1-\alpha)|U_R|}.
\]
\end{cor}

\begin{proof}
For a local ring,
\[
|U_{x}| = \frac{|U_R|}{|1+\ann(x)| \cap |U_R|} = \frac{|U_R|}{|\ann(x)|},
\]
which implies $|\ann(x)| |U_{x}| = |U_R|$ for all $x \in R$.
\end{proof}

\begin{rem}
\label{rem:irreversible}
Although the stationary distribution has a simple product structure, note that the Markov chains $(X_n)_{n \in \mathbb{Z}_+}$ and $(\Xu n)_{n \in \mathbb{Z}_+}$ are not reversible (see Definition~\ref{def:revers}). We illustrate this by comparing the stationary probabilities of the entries $0$ and $1$ in a finite chain ring for $(\Xu n)_{n \in \mathbb{Z}_+}$. Using Corollary~\ref{cor:stat-prob-units-special},
the ratio of the transitions between $1$ and $0$ are given by
\[
\frac{\mathbb{P}(0 \to 1)}{\mathbb{P}(1 \to 0)} = \frac{\alpha/\r}{\alpha/\r + (1-\alpha) \beta_0} = \frac{\alpha}{\alpha + (1-\alpha) \beta_0 \r}.
\]
but this is not equal to the ratio $\pi(1)/\pi(0)$.
\end{rem}

\subsection{Finite chain rings}
\label{sec:stat-dist-finite-chain}

It turns out that the stationary distribution can be described completely in the case of finite chain rings. We refer to Section~\ref{sec:finite-chain-ring} for terminology on finite chain rings. The poset of ideals of $R$ is a chain of height $k$. Every nonzero element $x$ in $R$ belongs to some $S_i$ for $0 \leq i \leq k$. 

\begin{thm} 
\label{thm:stat-dist-finite-chain}
The stationary distribution $\pi(x)$ for $x \in R$ 
in the chain $(\Xu n)_{n \in \mathbb{Z}_+}$ is given by
\begin{equation}\label{statdistpk}
\pi(x) = \begin{cases}
\ds \frac{\alpha}{q^{k-i-1}(1+(q-1)\,\alpha)^{i+1}}, & \text{ if $x \in S_i$ with $i<k$,} \\[0.5cm]
\ds \frac{1}{(1+(q-1)\,\alpha)^k},  & \text{if $x=0$.}
\end{cases}
\end{equation}
\end{thm}

\begin{proof}
Since finite chain rings are also local, we use Corollary~\ref{cor:stat-dist-local-ring}. In this case, $\phi$ can be identified with $\{0,\dots,k\}$ with $0$ corresponding to units and $k$ to the zero element. For $i,j \in \phi$, $I_i \subsetneq I_j$ if and only if the corresponding integers satisfy $j<i$.
The case $i=0$ is already covered by Corollary~\ref{cor:stat-prob-units-special}. We prove the other cases for $i \leq k-1$ by induction. We obtain, for $x \in S_i$,
\begin{align*}
\pi(x) &= \frac{\alpha + (1-\alpha)u\sum_{j<i, y \in S_j} \pi(y)}
{q^{k-1}(1+(q-1)\,\alpha)}, \\
&= \pi(1) + \frac{(1-\alpha)(q-1)}{(1+(q-1)\,\alpha)} 
\sum_{j=0}^{i-1} \frac{\alpha}{q^{k-j-1}(1+(q-1)\,\alpha)^{j+1}},
\end{align*}
by the induction assumption. This is now a geometric series, which is easily summed to obtain the desired result. The case of $\pi(0)$ can be then explicitly evaluated again using Corollary~\ref{cor:stat-dist-local-ring}.
\end{proof}

\section{Mixing Time}
\label{sec:mixing-zpk}

As described in Section~\ref{sec:def}, irreducible and aperiodic Markov chains converge to their unique stationary distribution.  
In this section, we will be interested in the speed of this convergence. It is well-known (see, for example \cite[Theorem 4.9]{LevinPeresWilmer}) that the convergence is exponentially fast. But we would like to know how the constant in the exponent scales with the size of the ring. 
We will give an elementary probabilistic argument proving that the mixing time is a constant independent of the size of the ring for our most general Markov
chain $(X_n)_{n \in \mathbb{Z}_+}$.

We begin with the relevant definitions. Define a natural metric on the space of probability distributions on $\Omega$ as follows.
\begin{defn}
The {\em total variation distance} between two probability distributions $\mu$ and $\nu$ on $\Omega$ is given by
\[
|| \mu - \nu ||_{\text{TV}} = \frac{1}{2} \sum_{x \in \Omega} |\mu(x) - \nu(x)|.
\]
\end{defn}
Suppose we start the Markov chain at some $x \in \Omega$. Then we obtain for each $n \in \mathbb{N}$, a probability distribution on $\Omega$ simply by evolving the chain, which we call $M^n(x,\cdot)$. We will denote the distance at time $n$ between this distribution, maximized over $x$, and $\pi$ by
\begin{equation}
\label{tvdist}
d(n) = \max_{x \in \Omega} ||M^n(x,\cdot) - \pi(\cdot) ||_{\text{TV}}.
\end{equation}
Fix an $\epsilon < 1/2$ for technical reasons.

\begin{defn}
\label{def:mixing-time}
The {\em mixing time} of a Markov chain $(Y_n)_{n \in \mathbb{Z}_+}$ with stationary distribution $\pi$ is given by
\[
t_{\text{mix}}(\epsilon) = \min \{n \mid d(n) \leq \epsilon \}.
\]
\end{defn}

Roughly speaking, the mixing time is at least as large as the relaxation time
(see Definition~\ref{def:relax-time}). 
The precise apriori bounds for reversible chains are given in \cite[Theorems 12.3 and 12.4]{LevinPeresWilmer}. 
For reversible Markov chains (see Definition~\ref{def:revers}), there are an abundance of techniques to compute the mixing time \cite{aldous-fill-2002,LevinPeresWilmer}. As we have shown in Remark~\ref{rem:irreversible}, $(X_n)_{n \in \mathbb{Z}_+}$ is not reversible. 
However, we will be able to use coupling techniques to establish our main result. 

\begin{defn}
A {\em coupling of Markov chains} with transition matrix $M$ is a process 
$(Y_n, Z_n)_{n \in \mathbb{Z}_+}$ with the property that both $(Y_n)_{n \in \mathbb{Z}_+}$ and $(Z_n)_{n \in \mathbb{Z}_+}$ are Markov chains with transition matrix $M$ (with possibly different starting distributions).
\end{defn}

Let $(Y_n, Z_n)_{n \in \mathbb{Z}_+}$ be a coupling and 
$\tauc$ be the first time the chains meet, i.e.
\begin{equation}
\label{coupling-time}
\tauc = \min \{n \mid Y_n = Z_n \}.
\end{equation}
Let $\mathbb{P}_{y,z}$ be the probability for the coupling where $Y_0 = y$ and $Z_0 = z$. The usefulness of coupling is that knowledge of $\tauc$ gives a useful bound for the mixing time.  
The precise result that we will use is the following.

\begin{thm}[{\cite[Corollary 5.3]{LevinPeresWilmer}}]
\label{thm:coupling-time}
Let $(Y_n, Z_n)_{n \in \mathbb{Z}_+}$ be a coupling and $\tauc$ be the coupling time as defined in \eqref{coupling-time}. Then
\[
d(n) \leq \max_{y,z \in \Omega} \mathbb{P}_{y,z} (\tauc > n).
\]
\end{thm}

We are now in a position to prove our mixing time bound.

\begin{proof}[Proof of Theorem~\ref{thm:mixing-time}]
\footnote{We are grateful to M. Krishnapur for suggesting this proof.}
We now describe the coupling for our Markov chain $(X_n)_{n \in \mathbb{Z}_+}$ that will prove our result. Let $(X^{(1)}_n, X^{(2)}_n)_{n \in \mathbb{Z}_+}$ be a coupling of two samples of $(X_n)_{n \in \mathbb{Z}_+}$ starting at $x^{(1)}_0, x^{(2)}_0 \in R$ respectively. 

Suppose we have run the joint chain up to time $n$ and they have not yet coupled.
We first toss a common coin with Heads probability $\alpha$ for both samples. If the coin lands Tails, we choose two independent elements $y_1, y_2$ according to the $\beta$ distribution defined in \eqref{cond-dist} and set $x^{(1)}_{n+1} = x^{(1)}_n \times y_1$, $x^{(2)}_{n+1} = x^{(2)}_n \times y_2$. That is, both chains move independently.
If the coin land Heads, we sample a uniformly random element $z \in R$. We then set $x^{(1)}_{n+1} = x^{(1)}_n + z$ and $x^{(2)}_{n+1} = x^{(2)}_n + \left(x^{(1)}_n + z - x^{(2)}_n \right)$. This is a valid coupling because $\left(x^{(1)}_n + z - x^{(2)}_n \right)$ is uniformly random in $R$ if $z$ is. At this point, 
$x^{(1)}_{n+1} = x^{(2)}_{n+1}$. It is easy to ensure that both $X^{(1)}$ and $X^{(2)}$ remain coupled for all future time by performing the same procedure for both at each future step.

As a consequence of this coupling procedure, the probability that $X^{(1)}$
and $X^{(2)}$ do not remain coupled up to time $n$ is a geometric random variable with success probability $\alpha$. That is, $\mathbb{P}(\tauc = n) = (1 - \alpha)^{n-1} \alpha$ for $n \in \mathbb{N}$.
Supposing that $\alpha < 1$, we thus obtain
\[
\mathbb{P}_{x^{(1)}_0, x^{(2)}_0} (\tauc > n) \leq (1 - \alpha)^n.
\]
The right hand side is independent of the initial conditions, and we obtain from Theorem~\ref{thm:coupling-time} that $d(n) \leq (1 - \alpha)^n$.
From Definition~\ref{def:mixing-time}, we find
\[
t_{\text{mix}}(\epsilon) \leq \left\lceil \frac{\log \epsilon}{\log (1 - \alpha)} \right\rceil.
\]
In the extreme case that $\alpha$ is equal to $1$, the Markov chain is the random walk on the complete graph on $|R|$ vertices. In that case, it is well-known that it mixes in one step. These two cases can be unified by adding an extra step, completing the proof.
\end{proof}

\section{Open Questions}
\label{sec:open}

In this work, we have studied algebraic and probabilistic properties of two natural Markov chains on a finite commutative ring. When the multiplication probabilities are uniform, several pertinent questions about the stationary distribution remain unanswered.
In particular, one can consider the least common denominator of the stationary probabilities, informally called the {\em partition function}. For instance, the partition function for the finite chain rings studied in Section~\ref{sec:stat-dist-finite-chain} is given, using Theorem~\ref{thm:stat-dist-finite-chain}, by 
\[
q^{k-1} (1 + (q-1) \alpha)^k.
\]
In all the cases that we have looked at, the partition function factorizes completely in terms of factors linear in $\alpha$. Why this factorization happens is an open question.
A natural class of rings for which more refined results should be available are the integer rings $\Z{m}$. In the case of squarefree integers, we have the following empirical observation. Suppose $m = p_1 \cdots p_k$, where $p_i$'s are primes. For $S \subset \{1,\dots,k\}$, let $m_S = \prod_{i \in S} p_i$ and $u_S = \prod_{i \in S} (p_i-1)$. Then the partition function for $(\Xu n)_{n \in \mathbb{Z}_+}$ on $\Z{m}$ seems to be
\[
\prod_{\emptyset \neq S \subset \{1,\dots,k\}} \left( m_S - u_S + u_S \alpha \right).
\]
We have proved analogous results about similar Markov chains on noncommutative rings have appeared in \cite{AS-2018}. The determination of the partition function for such chains is completely open.

In our proof of the upper bound for the mixing time, we have only used the additive structure of the ring. It is likely that one can prove even faster mixing by taking into account the multiplicative transitions. It might be an interesting problem to understand this mixing better.

\section*{Acknowledgements}
We are very grateful to the anonymous referees for many constructive suggestions.
We would also like to thank M. Krishnapur and B. Steinberg for enlightening discussions.
The authors would like to acknowledge support in part by a UGC Centre for Advanced Study grant. 
The first author (AA) would like to acknowledge support from Department of Science and Technology grants DST/INT/SWD/VR/P-01/2014 and EMR/2016/006624.

\bibliographystyle{plain}
\bibliography{ringpapers}

\end{document}